\documentclass[authoryear]{elsarticle}
\usepackage{amsfonts,amssymb,amsmath,amsthm}
\usepackage{url}
\usepackage{scrextend}
\usepackage{enumerate}
\usepackage{enumitem}
\usepackage{xy}
\usepackage[toc,page]{appendix}
\usepackage{color}
\usepackage{nameref}
\usepackage{algpseudocode}
\usepackage{algorithm}
\usepackage[margin=3cm]{geometry}
\usepackage[dvipsnames]{xcolor}
\usepackage[colorlinks=true,linkcolor=RoyalBlue,citecolor=PineGreen]{hyperref}
\usepackage{comment}
\xyoption{all}
\usepackage{array}

\theoremstyle{plain}
\newtheorem{theorem}{Theorem}[section]
\newtheorem{lemma}[theorem]{Lemma}

\newtheorem{proposition}[theorem]{Proposition}
\newtheorem{definition}[theorem]{Definition}

\newtheorem{example}[theorem]{Example}
\newtheorem{remark}[theorem]{Remark}

\newtheorem*{acknowledgment}{Acknowledgment}
\newtheorem*{remark-non}{Remark}

\newtheorem{innercustomthm}{Main Result}
\newenvironment{customthm}[1]
  {\renewcommand\theinnercustomthm{#1}\innercustomthm}
  {\endinnercustomthm}

\theoremstyle{definition}

\DeclareMathOperator{\height}{height}
\DeclareMathOperator{\Output}{\textbf{Output}}
\DeclareMathOperator{\Input}{\textbf{Input}}
\DeclareMathOperator{\unmixed}{\textbf{unmixed}}

\DeclareMathOperator{\Ap}{Ass}
\DeclareMathOperator{\ggcd}{ggcd}
\DeclareMathOperator{\pinvert}{\textbf{pinvert}}
\DeclareMathOperator{\lc}{lc}
\DeclareMathOperator{\Res}{Res}
\DeclareMathOperator{\prem}{sprem}
\DeclareMathOperator{\Rep}{Rep}
\DeclareMathOperator{\uupper}{up}
\DeclareMathOperator{\llower}{low}
\DeclareMathOperator{\adj}{adj}

\DeclareMathOperator{\lead}{lead}

\author[cuny]{E.~Amzallag}
\ead{eamzallag@gradcenter.cuny.edu}

\author[cuny]{M.~Sun}
\ead{msun@gradcenter.cuny.edu}

\author[ia]{G.~Pogudin}
\ead{pogudin.gleb@gmail.com}

\author[tdtu]{T.~N.~Vo\corref{cor}}
\ead{vongocthieu@tdtu.edu.vn}

\cortext[cor]{Corresponding author}

\address[cuny]{CUNY Graduate Center, Department of Mathematics, NY 10016, USA}

\address[ia]{Institute for Algebra, Johannes Kepler University, 4040, Linz, Austria}

\address[tdtu]{Faculty of Mathematics and Statistics, Ton Duc Thang University, Ho Chi Minh City, Vietnam}

\tnotetext[t1]{Gleb Pogudin's current address: Courant Institute of Mathematical Sciences, New York University, USA.}

\begin{document}



\title{Complexity of Triangular Representations of Algebraic Sets\tnoteref{thanks}}
\tnotetext[thanks]{Work was supported by the strategic program ``Innovatives O\" O 2020'' by the Upper Austrian Government, by Austrian Science Fund FWF grant Y464-N18, by the NSF grants CCF-095259, CCF-1563942, DMS-1606334, by the NSA grant \#H98230-15-1-0245, by CUNY CIRG \#2248, and by PSC-CUNY grant \#69827-00 47.}


\begin{abstract}
Triangular decomposition is one of the standard ways to represent the radical of a polynomial ideal.
A general algorithm for computing such a decomposition was proposed by A.\,Sz\'ant\'o.
In this paper, we give the first complete bounds for the degrees of the polynomials and the number of components
in the output of the algorithm, providing explicit formulas for these bounds.
\end{abstract}
\maketitle

\section{Introduction}

The general problem considered in the paper is: given polynomials {$f_0, \ldots, f_r \in k[x_1, \ldots, x_n]$}, where $k$ is a computable subfield of $\mathbb{C}$, 
represent the set of all polynomials vanishing on the set of solutions of the system $f_0 = \ldots = f_r = 0$.
This set of polynomials is called \emph{the radical} of the ideal generated by $f_0, \ldots, f_r$.
The problem is important for computer algebra and symbolic computations, as well as for their applications 
(for example, \citep{OvchinnikovPogudinVo,BurgisserSchei}).
Several techniques can be used to solve the problem; for example, Gr\"obner bases, geometric resolution, and triangular decomposition.
Representing the radical of an ideal is an intermediate step in many other algorithms.
Thus, it is crucial to understand the size of such a representation, as the size affects the complexity of the further steps.
The size of the representation can be expressed in terms of a degree bound for the polynomials appearing in the representation and their number.
The main result of the paper is the first complete bound on the degrees (Theorem~\ref{thm:degreeBound}) and the number of components (Theorem~\ref{thm:BoundComponents})
for the algorithm designed by A.\,Sz\'ant\'o in~\citep{SzantoThesis} for computing a triangular decomposition.

For Gr\"obner bases, a bound which is doubly-exponential in the number of variables is given in~\citep{Laplagne}.
Moreover, an example constructed in~\citep{ChistovDoubExpLowBound} shows that there are ideals such that every set of generators of
the radical (even those sets that are not Gr\"obner bases) contains a polynomial of doubly-exponential degree.
Geometric resolution and triangular decomposition do not represent the radical via its generators, so it was hoped that these representations might have better degree bounds.
For geometric resolution, singly-exponential degree bounds were obtained in~\citep{GiustiLecerf,Lecerf00,Lecerf03}
(for prior results in this direction, see references in~\citep{Lecerf03}).

Algorithms for triangular decomposition were an active area of research during the last two decades.
Some results of this research were tight degree upper bounds for a triangular decomposition of an algebraic variety
given that the decomposition is irredundant~\citep{Schost,SchostDahan}, an efficient algorithm for zero-dimensional
varieties~\citep{Lifting}, and implementations~\citep{Epsilon,RegularChains}.
However, to the best of our knowledge, there are only a few algorithms~\citep{GalloMishra,SzantoThesis,Schost} 
for computing triangular decomposition with proven degree upper bounds for the output.
The algorithms in~\citep{Schost} and~\citep{GalloMishra} have restrictions on the input polynomial system.
The algorithm in~\citep{Schost} requires the system to define an irreducible variety.
The algorithm in~\citep[Theorem~4.14]{GalloMishra} produces a characteristic set of an ideal, which represents the radical of the ideal only if the ideal is characterizable~\citep[Definition~5.10]{Hubert1} (for example, an ideal defined by $x_1x_2$ is not characterizable).
Together with~\citep[Proposition~5.17]{Hubert1} this means that the algorithm from~\citep{GalloMishra} represents the radical of an ideal if the radical can be defined by a single regular chain.

The algorithm designed by~\cite{SzantoPaper,SzantoThesis} does not have any restrictions on the input system.
However, it turns out that the argument in~\citep{SzantoThesis} 
does not imply the degree bound $d^{O(m^2)}$ ($m$ is the maximum codimension of the components of the ideal, $d$ is a bound for degrees of the input polynomials) stated there.
The reason is that the argument in~\citep{SzantoThesis} did not take into account possible redundancy of the output (see Remark~\ref{rem:IncreasingDegreeBound}).
Moreover, in Example~\ref{ex:redundant} we show that the sum of degrees of 
extra components produced by the algorithm can be significantly 
larger than the degree of the original variety.
In this paper, we take these extra components into account and prove
an explicit degree bound of the form $d^{\mathrm{O}(m^3)}$ for the algorithm.
More precisely, we prove that:

\begin{customthm}{}[Theorem \ref{thm:degreeBound}]
Let $f_0, \ldots, f_r \in k[x_1, \ldots, x_n]$ be polynomials with $\deg f_i \leq d$ for all $0 \leq i \leq r$ ($d > 1$).
Assume that the maximum codimension of prime components of the ideal $(f_0, \ldots, f_r)$ is $m \geq 2$, and $r \leq d^m$.
Then the degree of any polynomial $p$ appearing in the output of Sz\'ant\'o's algorithm or during the computation does not exceed
\begin{align*}
\deg(p) \leq n d^{(\frac{1}{2}+\epsilon) m^3}
\end{align*}
where  $\epsilon$ is some decreasing function of $m, d$ and $\epsilon$ is bounded by $5$
(for a more general statement, we refer to Section~\ref{sec:degrees}).
\end{customthm}

\begin{customthm}{}[Theorem \ref{thm:BoundComponents}]
Let $F \subset k[x_1,\ldots,x_n]$ be a finite set of polynomials of degree at most $d$. Let $m$ be the maximum of codimension of prime components of $\sqrt{(F)} \subseteq k[x_1,\ldots,x_n]$. 
Then the number of squarefree regular chains in the output of Sz\'ant\'o's algorithm applied to $F$ is at most
$$\binom{n}{m}\left( (m + 1)d^m+1\right)^m.$$
\end{customthm}


\section{Preliminaries}
\label{sec:preliminaries}

Throughout the paper, all fields are of characteristic zero and all logarithms are binary.

Throughout this section, let $R = k[x_1, x_2, \ldots, x_n]$, where $k$ is a field.  
We fix an ordering on the variables $x_1 < x_2 < \dots < x_n$.  
Consider a polynomial $p \in R$.
We set $\height(p):= \displaystyle \max_i\deg_{x_i}(p)$.  
The highest indeterminate appearing in $p$ is called its leader and will be defined by $\lead(p)$.
By $\text{lc}(p)$ we denote the leading coefficient of $p$ when $p$ is written as a univariate polynomial in $\lead(p)$.
\begin{definition}
  Given a sequence $\Delta = (g_1, g_2, \ldots, g_m)$ in $R$,  
  we say that $\Delta$ is a \emph{triangular set} if $\lead(g_i) < \lead(g_j)$ for all $i<j$.
\end{definition}

\begin{remark}
Note that any subsequence of a triangular set is a triangular set.  In what follows, the subsequences of $\Delta$ of particular interest are the ones of the form $\Delta_j := (g_1, g_2, \ldots, g_j)$, $1 \leq j \leq m$ and $\Delta_0 := \varnothing$.
\end{remark}

Triangular sets give rise to ideals via the following notion.
\begin{definition}
Let $f, g \in R$ with $\lead(g) = x_j$.  
We consider $f$ and $g$ as univariate polynomials in $x_j$ with the coefficients from the field $k(x_1, x_2, \ldots, x_{j-1}, x_{j+1}, \ldots, x_n)$ and let $f = \tilde{q}g + \tilde{r}$ be the result of univariate polynomial division of $f$ by $g$ with coefficients in this field.
Let $\alpha$ be the smallest nonnegative integer such that $g := \lc{(g)}^\alpha \tilde{g}$ and $r := \lc{(g)}^\alpha \tilde{r}$ are polynomials, so we obtain an equation
\begin{align*}
\lc{(g)}^{\alpha}f = q g + r
\end{align*}
with $q, r \in R, \deg_{x_j}(r) < \deg_{x_j}(g), \alpha \in \mathbb{N}$.  
One can show that $\alpha \leq \deg_{x_j}(f) - \deg_{x_j}(g) + 1$.  
For uniqueness of $q, r$, we require $\alpha$ to be minimal.  
We say that $r$ is \emph{pseudoreminder of $f$ by $g$} and denote it by $\prem(f,g)$.
\end{definition}

\begin{definition}
Let $\Delta = (g_1, g_2, \ldots, g_m)$ be a triangular set and let $f \in R$.  
The \emph{pseudoremainder of $f$ with respect to $\Delta$} is the polynomial $f_0$ in the sequence $f_m = f, f_{s-1} = \prem(f_{s}, g_{s}), 1 \leq s \leq m$.  We denote $f_0$ by $\prem(f, \Delta)$.

We say that $f$ is \emph{reduced with respect to $\Delta$} if $f=\prem(f, \Delta)$.
\end{definition}

\begin{remark}
The computation of the pseudoremainder of $f$ with respect to $\Delta$ gives rise to the equation
\begin{align*}
\lc{(g_m)}^{\alpha_m}\dots \lc{(g_1)}^{\alpha_1}f = \sum_{s=1}^m q_sg_s+f_0
\end{align*}
where each $\alpha_s \leq \deg_{\lead{(g_s)}}(f_s) - \deg_{\lead (g_s)}(g_s) + 1$.
\end{remark}

\begin{definition}
Given a triangular set $\Delta$ in $R$, we define the ideal
\[
\Rep(\Delta) := \{ p \in R \,|\, \exists N \,:\, H^Np \in \langle \Delta \rangle\}, \text{ where } H := \lc(g_1)\ldots \lc(g_m).
\]
We say that a triangular set $\Delta \subset R$ \emph{represents} an ideal $I$ if $I = \Rep(\Delta)$.
\end{definition}

\begin{definition}\label{def:ap}
For an ideal $I \subset R$, we consider the irredundant prime decomposition $\sqrt{I} = I_1 \cap \ldots \cap I_r$ of its radical.   
We call the $I_1, \ldots, I_r$ \emph{the associated primes} of $I$ and denote the set of associated primes of $I$ by $\Ap(I)$.  
 When $I = \Rep(\Delta)$, we will write $\Ap(\Delta)$ instead of $\Ap(I)$.

We say that $\sqrt{I}$ and the corresponding variety $V(I)$ are \emph{unmixed} if all the associated prime ideals have the same dimension.
\end{definition}

\begin{definition}
  Let $\Delta = (g_1, g_2, \ldots, g_m)$ be a triangular set of $R$
with $I = \Rep(\Delta)$ 
and, for each $1 \leq i \leq m-1$, let $\{P_{i,j}\}_{j=1}^{r_i}$ be the prime ideals in the irredundant prime decomposition of the radical of $\Rep(\Delta_i)$.  
\begin{enumerate}[label=(\alph*),itemsep=2pt]
\item\label{cond:a}
if $\text{lc}(g_{i+1}) \notin  P_{i,j}$ for every for every $1 \leq i \leq m - 1$ and $1 \leq j \leq r_i$, then $\Delta$ is called \emph{a regular chain}, see~\cite[Definition~5.7]{Hubert1}.
\item\label{cond:b}
if $g_{i+1}$ is square-free over $K(P_{i,j}) := \text{Quot}(R/P_{i,j})$ for every $1 \leq j \leq r_i$ and $1 \leq j \leq r_i$, then $\Delta$ is called \emph{a squarefree regular chain}, see~\cite[Definition~7.2]{Hubert1} \\
(Here, $\text{Quot}(R/P_{i,j})$ is the field of fractions of $R/P_{i,j}$.)
\end{enumerate}
\end{definition}




\begin{theorem}[see {{\cite[Proposition~2.7]{BurgisserSchei}}}]
  If $\Delta$ is a regular chain, then $\Rep(\Delta) = \{h \in R \,|\, \prem(h,\Delta)=0 \}$ and all of the prime ideals in the irredundant prime decomposition of $\Rep(\Delta)$ have the same dimension.
\end{theorem}

\begin{theorem}[see {{\cite[Corollary~7.3]{Hubert1}}}]
  If $\Delta$ is a squarefree regular chain, then $\Rep(\Delta)$ is a radical ideal.
\end{theorem}

\begin{remark}
We use terminology different from the one used in~\citep[Section~2.4.3]{SzantoThesis}.
The correspondence between these two terminologies is the following: a regular chain is called a weakly unmixed triangular set in~\citep{SzantoThesis} and a squarefree regular chain is called an unmixed triangular set in~\citep{SzantoThesis}.
\end{remark}

Now we are ready to define the main object we will compute.

\begin{definition}
The \emph{triangular decomposition} of an ideal $I \subset R$ is a set $\{\Delta_1, \ldots, \Delta_s\}$ of squarefree regular chains such that 
\[
  \sqrt{I} = \bigcap\limits_{i = 1}^s \Rep(\Delta_i).
\]
\end{definition}

In the rest of the section, we introduce notions and recall results about computing modulo a triangular set.

\begin{definition}\label{def:mult_table}
Let $\Delta = (g_1, \ldots, g_m)$ be a triangular set in $R$ with $\lead{(g_s)} = x_{l+s}$ and $d_s := \deg_{x_{l+s}}(g_s)$ for every $1 \leq s \leq m$, where $l := n - m$.  
We define
\begin{itemize}[itemsep=3pt]
\item
$A(\Delta):=k(x_1, x_2, \ldots, x_l)[x_{l+1}, \ldots, x_n]/( \Delta )_{k(x_1, x_2, \ldots, x_l)}$, where the subscript reminds us that we treat elements of the field $k(x_1, x_2, \ldots, x_l)$ as scalars and consider the quotient $A(\Delta)$ as an algebra over this field.
\item
The \emph{standard basis} of $A(\Delta)$, which we will denote by $B(\Delta)$, is the set
\[
B(\Delta) := \{x_{l+1}^{\alpha_1}\dots x_n^{\alpha_m} \mid 0 \leq \alpha_s < d_s, 1 \leq s \leq m\}.
\]
\item
The set of \emph{structure constants} of $A(\Delta)$ is the collection of the coordinates of all products of pairs of elements of $B(\Delta)$ in the basis $B(\Delta)$.  These structure constants may be organized into a table, which we will refer to as the \emph{multiplication table} for $A(\Delta)$ and which we will denote by $M(\Delta)$.
\item
The \emph{height of the structure constants of $A(\Delta)$} is the maximum of the heights of the entries of $M(\Delta)$.  We denote this quantity by $\Gamma(\Delta)$ or $\Gamma$ when the triangular set under consideration is clear from context.  We will also use the notation $\Gamma_j$ for $\Gamma(\Delta_j)$.
\item
An element of $ A(\Delta)$ is called \emph{integral} if its coordinates in the standard basis $B(\Delta)$ belong to~$k[x_1,\ldots,x_l]$.
\end{itemize}
\end{definition}

\begin{proposition} [{see \citep[Prop. 3.3.1, p.76]{SzantoThesis}}] \label{prop:HtProd}
Let $\Delta$ be a triangular set and let $a_1, a_2, \ldots, a_k$ be elements of $A(\Delta)$ with heights at most $d$.  Moreover, assume that the denominators of the coordinates of $a_1, a_2, \ldots, a_k$ in the basis $B(\Delta)$ divide
$\displaystyle \prod_{s=1}^m \lc(g_s)^{\beta_s} \hspace{.1cm}$ and also assume that
$\displaystyle \hspace{.1cm} \sum_{s=1}^m \beta_s \cdot \height(\lc(g_s)) \leq d'$.
Then 
\begin{itemize}[itemsep=2pt]
\item
$\height(a_1a_2) \leq \height(a_1)+\height(a_2)+2(d'+\Gamma)$ and
\item
$\height(a_1a_2\dots a_k) \leq kd+k\log{k}(d'+\Gamma).$
\end{itemize}
\end{proposition}
In Proposition~\ref{prop:HtProd}, if $a_1,\ldots,a_k$ are integral elements, then $\beta_1=\ldots=\beta_s=0$. In this case, one can choose $d'=0$.
We will also use denominator bounds in reducing an element modulo~$\Delta$.  

\begin{lemma} \label{lem:denominator}
Let $\Delta := (g_1,\ldots,g_m) \subset k[x_1,\ldots,x_n]$ be a squarefree regular chain such that $\height (g_s) \leq d$ for all $s=1,\ldots,m$. Let $f \in k[x_1,\ldots,x_n]$ be a polynomial of height at most $t$. Then there exist $\alpha_1, \ldots, \alpha_m \in \mathbb{N}$ and $q_1,\ldots,q_m,r \in k[x_1,\ldots,x_n]$ such that:
\begin{itemize}[itemsep=2pt]
\item $\lc (g_1)^{\alpha_1} \cdot \dots \cdot \lc (g_m)^{\alpha_m} \cdot f = q_1g_1 + \dots + q_m g_m +f_0$,
\item $f_0$ is reduced modulo $\Delta$, and
\item $\alpha_s \leq t (d+1)^{m-s}, \hspace*{.1cm} s=1, 2, \ldots, m$.
\end{itemize}
\end{lemma}
\begin{proof} Similar to \citep[Lemma 3.7]{BurgisserSchei}.
\end{proof}

\begin{remark}
Gallo and Mishra gave a bound in \cite[Lemma~5.2]{GalloMishra} for the degree of the pseudoremainder $f_0$.
We compare that bound with the corresponding bound on $f_0$ that can be derived from Lemma~\ref{lem:denominator}.  
In the table below, OB stands for ``Our Bound'' and GM stands for ``Gallo-Mishra.''

\begin{center}
\begin{tabular}{ |c|c|c| } 
 \hline
& $\height(g_s) \leq d \;\&\; \height(f) \leq t$ & $\deg(g_s) \leq d \;\&\; \deg(f) \leq t$ \\
\hline
 $\deg(f_0)$ & OB: $nt(d+1)^m$ & OB: $nt(d+1)^m$ \\ & GM: $(nt+1)(nd+1)^m$ & GM: $(t+1)(d+1)^m$ \\
\hline
 $\height(f_0)$ & OB: $t(d+1)^m$ & OB: $t(d+1)^m$ \\
 & GM: $(nt+1)(nd+1)^m$ & GM: $(t+1)(d+1)^m$ \\
 \hline
\end{tabular}
\end{center}

We see that the only case in which the bound from~\citep[Lemma~5.2]{GalloMishra} is smaller than the corresponding one derived from Lemma~\ref{lem:denominator} is represented by the upper-right cell, in which solely degrees are considered.  
In fact, \cite{GalloMishra} analyzes the complexity of the Ritt-Wu Characteristic Set Algorithm in terms of degrees.  
So our pseudoremainder bound cannot be used to improve their complexity analysis and vice versa, 
as can be seen by examining the lower-left cell in which heights are the focus.
\end{remark}

\section{Outline of Sz\'ant\'o's algorithm}
In this section, we recall main steps of the algorithm in \citep{SzantoThesis} for computing a triangular decomposition for a given algebraic set. The main algorithm is described in \citep[Theorem~4.1.7, p.~118]{SzantoThesis} and its proof. 

\begin{algorithm}[H]
\caption{Triangular decomposition 
algorithm}\label{alg:SzantoAlgorithm}

\begin{description}[leftmargin=2.5em,style=multiline]
  \item[In] A set of polynomials $F=\{f_0,f_1,\ldots,f_r\} \subset k[x_1,\ldots,x_n]$.
  
  \item[Out] A set $\Theta(F)$ of squarefree regular chains such that
\[
\sqrt{\left < F \right >}= \bigcap\limits_{\Delta \in \Theta}{\Rep(\Delta)}.
\]
\end{description}

\begin{enumerate}[leftmargin=!,labelwidth=1.5em, label=(\alph*),itemsep=3pt]
	\item\label{alg_one} For every $\mathbf{i} \subsetneq \{1, \ldots, n\}$, compute a regular chain $\Delta_{\mathbf{i}}$ with leaders $\{x_j | j \not\in \mathbf{i}\}$  
    such that for every prime component $P$ of $\sqrt{( F )}$
    \[
    \bigl( \dim (P) = |\mathbf{i}| \text{ and } P \cap k[x_i\mid i \in \mathbf{i}] = \{0\}\bigr) \Rightarrow \Rep(\Delta_{\mathbf{i}}) \subseteq P.
    \]
    For details, see \citep[Cor. 4.1.5, p. 115]{SzantoThesis}.
	
    \item\label{alg_two} For every $\mathbf{i} \subsetneq \{1, \ldots, n\}$, compute the multiplication table $M(\Delta_{\mathbf{i}})$ of the algebra $A(\Delta_{\mathbf{i}})$ (see Definition~\ref{def:mult_table}).
    
	\item\label{alg_three} For every $\mathbf{i} \subsetneq \{1, \ldots, n\}$, compute a set $\mathcal{U}(\Delta_{\mathbf{i}})$ of squarefree regular chains 
    \[ 
    \unmixed_{|\Delta_{\mathbf{i}}|}^{|\mathbf{i}|}(\Delta_{\mathbf{i}} ,M(\Delta_{\mathbf{i}}), f, 1), \text{ where } f := \sum\limits_{j=0}^{r} f_ix_{n + 1}^j
    \]
    using Algorithm~\ref{alg:Unmixed_m_l} below.
	\item\label{alg_four} \textbf{Return} $\Theta(F):=\bigcup\limits_{\mathbf{i} \subsetneq \{1, \ldots, n\}}{\mathcal{U}(\Delta_{\mathbf{i}})}.$
\end{enumerate}
\end{algorithm}

Step~\ref{alg_three} of Algorithm~\ref{alg:SzantoAlgorithm} uses function $\unmixed$ with the following full specification.
Parts concerning multiplication tables are technical and important only for efficiency.

\noindent
\hypertarget{specunmixed}{\textbf{Specification of $\unmixed^{l}_{m}$.}}

\begin{description}[leftmargin=1.5em,style=multiline]
  \item[In] 
  \begin{enumerate}[leftmargin=!,labelwidth=2.5em,itemsep=0cm]
  \item Nonnegative integers $m$ and $l$. We set $n := m + l$.
  \item A regular chain 
 $\Delta=\{g_1, \ldots, g_m\} \subset k[x_1, \ldots, x_n]$ such that for all $1 \leq s \leq m$
	\begin{itemize}[leftmargin=0.5em,itemsep=3pt]
		\item $\lead(g_s)=x_{l + s}$;
		\item $\lc (g_s) \in k[x_1,\ldots,x_l]$;
		\item $g_s$ is reduced modulo $\{g_1, \ldots, g_{s - 1}\}$.
	\end{itemize}
    \item The multiplication table $M(\Delta)$ of the algebra $A(\Delta)$, see Definition~\ref{def:mult_table}. 
    \item Polynomials $f$, $h$ in $k[x_1, \ldots, x_{n + c}]$ for some $c > 0$ reduced with respect to $\Delta$.
  \end{enumerate}
  
  \item[Out] \hspace{3mm} A set $\{(\Delta_1, M(\Delta_1)), \ldots, (\Delta_r, M(\Delta_r))\}$ such that
  \begin{itemize}[leftmargin=3em,itemsep=3pt]
    \item $\Delta_i$ is a squarefree regular chain in $k[x_1, \ldots, x_n]$ for every $1 \leq i \leq r$;
    \item $M(\Delta_i)$ is the multiplication table of the algebra $A(\Delta_i)$ for every $1 \leq i \leq r$;
    \item $\bigcup\limits_{i=1}^{r}{\Ap(\Delta_i)} = \{ P \in \Ap(\Delta) \,|\, {f} \equiv 0, \, {h} \not\equiv 0 \, \mod P \}$ (see Definition~\ref{def:ap});
	\item $\Ap(\Delta_i) \cap \Ap(\Delta_j) = \varnothing \,\; \forall\; i \neq j$.
  \end{itemize}
\end{description}

Before describing the algorithm itself, we will give some intuition behind it.

Informally speaking, the main goal of $\unmixed$ is to transform a single regular chain $\Delta$ into a set of regular chains $\Delta_1, \ldots, \Delta_r$ such that
\begin{enumerate}[label=(\alph*)]
  \item\label{item:sq_free} $\Delta_1, \ldots, \Delta_r$ are squarefree regular chains;
  \item\label{item:no_extra} prime components of $\bigcap\limits_{i = 1}^r \Rep(\Delta_i)$ are exactly the prime components of $\Rep(\Delta)$, on which $f$ vanishes and $h$ does not vanish.
\end{enumerate}
It is instructive first to understand how this transformation is performed in the univariate case, i.e. in the case when all regular chains consist of a single polynomial only.
This case is also discussed in~\citep[p.~124-125]{SzantoThesis}. 
Let $\Delta$ consist $g(x) \in k[x]$. 
A polynomial satisfying only property~\eqref{item:no_extra} can be computed using gcd's as follows
\begin{equation}\label{eq:univ_fh}
\frac{\gcd_x(g, f)}{\gcd_x(g, f, h)}.
\end{equation}
A set of polynomials satisfying only property~\eqref{item:sq_free} can be obtained by separating the roots of $g(x)$ according to their multiplicity again using gcd's
\begin{equation}\label{eq:univ_sqfree}
\frac{g \gcd_x(g, g', g'')}{\gcd^2_x(g, g')}, \frac{\gcd_x(g, g') \gcd_x(g, g', g'', g^{(3)})}{\gcd^2_x(g, g', g'')}, \ldots 
\end{equation}
Formulas~\eqref{eq:univ_fh} and~\eqref{eq:univ_sqfree} can be combined to yield to a set of polynomials satisfying both properties~\eqref{item:sq_free} and~\eqref{item:no_extra}:
\begin{equation}\label{eq:univ}
  q_i := \frac{\gcd_x(g, \ldots, g^{(i - 1)}, f) \gcd_x(g, \ldots, g^{(i + 1)}, f) \gcd^2_x(g, \ldots, g^{(i)}, f,  h)}{\gcd_x^2(g, \ldots, g^{(i)}, f) \gcd_x(g, \ldots, g^{(i - 1)}, f, h) \gcd_x(g, \ldots, g^{(i + 1)}, f, h)}, \;\; i = 1, 2, \ldots, \deg g.
\end{equation}

The generalization of this approach to the multivariate case is based on two ideas
\begin{enumerate}[label=(\alph*)]
  \item Perform the same manipulations with $g_m$ considered as univariate polynomials in $x_n$.
  \item Replace the standard univariate $\gcd$ with the \emph{generalized gcd} (denoted by $\ggcd$), that is a $\gcd$ modulo a regular chain $\Lambda := \{g_1, \ldots, g_{m - 1}\}$. Generalized gcds are described in~\citep[Lemma~3.1.3]{SzantoThesis}. Formula~\eqref{eq:univ} is replaced then by
  \begin{equation}\label{eq:ggcd}
    q_i := \frac{\ggcd_{x_n}(\Lambda, g_m, \ldots, g_m^{(i - 1)}, f) \ggcd_{x_n}(\Lambda, g_m, \ldots, g_m^{(i + 1)}, f) \ggcd^2_{x_n}(\Lambda, g_m, \ldots, g_m^{(i)}, f,  h)}{\ggcd_{x_n}^2(\Lambda, g_m, \ldots, g_m^{(i)}, f) \ggcd_{x_n}(\Lambda, g_m, \ldots, g_m^{(i - 1)}, f, h) \ggcd_{x_n}(\Lambda, g_m, \ldots, g_m^{(i + 1)}, f, h)}
  \end{equation}
  for $i = 1, 2, \ldots, \deg_{x_n} g_m$.
\end{enumerate}

Generalized gcd is always well-defined modulo a regular chain representing a prime ideal.
If the ideal represented by the regular chain is not prime, then generalized gcds modulo different prime components might have different degree, so it might be impossible to ``glue'' them together.
In order to address this issue, the $\unmixed$ function splits $\Rep(\Lambda)$ into a union of varieties represented by regular chains, over which all the generalized gcds in~\eqref{eq:ggcd} will be well defined.
Interestingly, this can be done by calling $\unmixed$ recursively, because the fact that some generalized gcd is well-defined and has degree $d$ can be expressed using equations and inequations.
These equations and inequations can be further combined with $f$ and $h$.

\begin{algorithm}[H]
\caption{Function $\unmixed_m^l(\Delta, M(\Delta), f, h)$}\label{alg:Unmixed_m_l}

Input and output are described in the~\hyperlink{specunmixed}{specification above}.

\begin{enumerate}[label=(\alph*),itemsep=3pt]
	\item If $m = 0$ (so $\Delta=\varnothing$), \textbf{return} $\varnothing$ if $f \neq 0$ or $h = 0$, and \textbf{return} $\{(\varnothing, \varnothing)\}$ otherwise
	\item Set $\Lambda:= \Delta_{m - 1} = \{g_1,\ldots,g_{m-1}\}$ and compute $M(\Lambda)$.
    \item\label{alg:unmixed_three} For every $1 \leq i \leq \deg_{x_{n}}g_m$ and every tuple $\mathbf{v} \in \mathbb{Z}_{\geqslant 0}^6$ with entries not exceeding $\deg_{x_{n}}g_m$, 
    compute
    a pair of polynomials $\phi_{i, \mathbf{v}}, \psi_{i, \mathbf{v}}$
    as described in~\citep[p.~128]{SzantoThesis}
    such that a system $\phi_{i, \mathbf{v}} = 0, \; \psi_{i, \mathbf{v}} \neq 0$ is equivalent to 
    \begin{itemize}
      \item $f = 0$ and $h \neq 0$,
      \item all six generalized gcds in~\eqref{eq:ggcd} are well-defined and their degrees are the entries of $\mathbf{v}$.
    \end{itemize}
    Formulas for $\phi_{i, \mathbf{v}}$ and $\psi_{i, \mathbf{v}}$ are given in the proof of Lemma~\ref{lem:Input(s)} and in~\citep[p.~128]{SzantoThesis}.
	\item  \label{alg:unmixed_four} For every pair $(\phi_{i,\mathbf{v}},\, \psi_{i,\mathbf{v}})$ computed in the previous step
    \begin{enumerate}[label=(\roman*)]
      \item Compute 
    \[\mathcal{L}_{i,\mathbf{v}} := \unmixed_{m - 1}^l(\Lambda, M(\Lambda), \phi_{i,\mathbf{v}},  \psi_{i,\mathbf{v}}).
    \]
      \item\label{out1} For every $(\Lambda_{i, \mathbf{v}},M(\Lambda_{i, \mathbf{v}})) \in \mathcal{L}_{i, \mathbf{v}}$ compute $q_{i, \mathbf{v}}$ using~\eqref{eq:ggcd} (more details in the proof of Theorem~\ref{thm:outputs} and in \citep[p.~129-130]{SzantoThesis}) 
      \item\label{out2} For every $q_{i, \mathbf{v}}$ computed in the previous step, add $(\Lambda_{i, \mathbf{v}} \cup \{ q_{i, \mathbf{v}} \}, M(\Lambda_{i, \mathbf{v}} \cup \{ q_{i, \mathbf{v}} \}))$ to the \textbf{output}
    \end{enumerate}
    \item \textbf{Return} the set of all pairs $(\Lambda_{i, \mathbf{v}} \cup \{ q_{i, \mathbf{v}} \}, M(\Lambda_{i, \mathbf{v}} \cup \{ q_{i, \mathbf{v}} \}))$ computed in the previous step 
	\end{enumerate}
\end{algorithm}

\begin{example}\label{ex:redundant}
  In this example, we will show that the output of Algorithm~\ref{alg:SzantoAlgorithm} can be redundant confirming~\citep[Remark~2.9]{BurgisserSchei}.
  We fix a positive integer $D$ and consider
  \begin{equation}\label{eq:redundant_example}
  F := \{ (x_1 - 1)(x_1 - 2)\ldots(x_1 - D)(x_2 - 1)(x_2 - 2)\ldots(x_2 - D) \}.
  \end{equation}
  Step~\ref{alg_one} of Algorithm~\ref{alg:SzantoAlgorithm} will output the following regular chains (see~\citep[Corollary~4.1.5]{SzantoThesis} for details)
  \begin{align*}
    &\Delta_{\{1\}} = \Delta_{\{2\}} = \{ (x_1 - 1)(x_1 - 2)\ldots(x_1 - D)(x_2 - 1)(x_2 - 2)\ldots(x_2 - D) \},\\
    &\Delta_{\varnothing} = \{ (x_1 - 1)(x_1 - 2)\ldots(x_1 - D)p_1(x_1), (x_2 - 1)(x_2 - 2)\ldots(x_2 - D)p_2(x_2) \},
  \end{align*}
  where $p_1(x_1)$ and $p_2(x_2)$ are additional factors, which can appear during the computation with Canny's generalized resultants (see~\citep[Proposition~4.1.2]{SzantoThesis}).
  
  At Step~\ref{alg_three} of Algorithm~\ref{alg:SzantoAlgorithm}, $\unmixed_2^0(\Delta_{\varnothing}, M(\Delta_{\varnothing}), f, 1)$ will be computed. 
  According to the \hyperlink{specunmixed}{specification} of $\unmixed$, the result of this computation will be a triangular decomposition of the set of common zeros of $\Rep(\Delta_{\varnothing})$ and $F$.
  Since the zero set of $\Rep(\Delta_{\varnothing})$ is finite, all these components are not components of the zero set of $F$.
  Points $\{(a_1, a_2) | a_1, a_2 \in \{1, 2, \ldots, D\}\}$ are common zeros of $\Rep(\Delta_\varnothing)$ and $F$, so the sum of the degrees of these extra components is at least $D^2$, and the degree of the zero set of $F$ is just $2D$.
  
  Moreover, this example can be generalized to higher dimensions by replacing~\eqref{eq:redundant_example} by
  \[
  F := \{ (x_1 - 1)(x_1 - 2)\ldots(x_1 - D) \ldots (x_n - 1)(x_n - 2)\ldots(x_n - D) \}.
  \]
  The degree of the zero set of $F$ is $nD$, but the sum of the degrees of extra components will be at least~$D^n$.
\end{example}


\section{Bounds for degrees}
\label{sec:degrees}

The following lemma is a refinement of \citep[Proposition~3.3.4, p. 75]{SzantoThesis}.
\begin{lemma} \label{lem:boundTheStrucConstant}
	Let $\Delta = (g_1, \ldots, g_m)$ be a squarefree regular chain such that $\height(g_s) \leq d$ for all $s$.  Suppose that for all $1 \leq s \leq m$ that
    \begin{enumerate}[itemsep=2pt]
    \item
    $\lead(g_s) = x_{l+s}$;
    \item
    $\text{lc}(g_s) \in k[x_1, \ldots, x_l]$;
    \item
    $g_s$ is reduced modulo $\Delta_{s-1}=(g_1, \ldots, g_{s-1})$, i.e.
    $\forall t<s,\,\,\, \deg_{x_{l+t}}(g_s) < \deg_{x_{l+t}}(g_t).$
    \end{enumerate}
    Then the height $\Gamma(\Delta)$ of the matrix $M(\Delta)$ of structure constants of $A(\Delta)$ (see Definition~\ref{def:mult_table}) does not exceed 
    $$(d+2)^{m+1}(\log(d+2))^{m-1}.$$
\end{lemma}

\begin{proof}
We first apply the matrix description of the pseudoremainder (see~\nameref{sec:appendix}) to products of the form $x_{l+1}^{e_1}x_{l+2}^{e_2}\dots x_{l+m}^{e_m}$, where $e_s \leq 2d_s-2$.  Note that these products are the ones considered in computing the structure constants for $A(\Delta)$ and that such a product will play the role of what we call $f$ in~\nameref{sec:appendix}.  
Also, what we called $g$ in the~\nameref{sec:appendix} will be $g_m$ in our application, as that is the first element we pseudo-divide by in reducing by $\Delta$.
We have two cases to consider: $e_m < d_m$ and $e_m \geq d_m$.

In the first case, the product of interest is already reduced modulo $g_m$ and so can itself be selected as the pseudoremainder by $g_m$.  So we can bound the height of its pseudoremainder by $\Delta$ by taking the maximum of $\Gamma_{m-1}:=\Gamma(\Delta_{m-1})$ and $d_m$.

In the second case, what we denote by $\mathbf{f}^{\llower}$ in the~\nameref{sec:appendix} is here a column vector with every entry 0 and what we denote by $\mathbf{f}^{\uupper}$ has exactly one nonzero entry, namely $x_{l+1}^{e_1}x_{l+2}^{e_2}\dots x_{l+m-1}^{e_{m-1}}$.

We first inspect the $G_0\cdot \adj(G_d)$ part of the pseudoremainder expression.  In computing this product, one will obtain a $d_m \times d_m$ matrix and each of its entries will be sum of products of at most $1+(d_m-1) = d_m$ reduced integral elements of $A(\Delta_{m-1})$.  (Note that we have products of reduced integral elements of $A(\Delta_{m-1})$ because $g_m$ is assumed to be reduced modulo $\Delta_{m-1}$.)

Completing the analysis of the number of multiplications needed to compute the pseudoremainder by $g_m$, we note that the product $x_{l+1}^{e_1}x_{l+2}^{e_2}\dots x_{l+m-1}^{e_{m-1}}$ can be split into two factors where the exponent of each $x_{l+s}$ is less than $d_s$ (because $e_s \leq 2d_s-2$).  So multiplying $G_0 \cdot \adj(G_d)$ by the column vector  $\mathbf{f}^{\uupper}$  
results in sums of products of at most $d_m+2$ reduced integral elements of $A(\Delta_{m-1})$.

	So by Proposition \ref{prop:HtProd} we have 
    $$\Gamma_m \leq (d_m+2)\cdot d+(d_m+2)\log(d_m+2)\cdot \Gamma_{m-1}.$$
We first replace $d_m$ by $d$ and estimate the first term as $(d+2)^2$ to obtain
$$\Gamma_s < (d+2)^2+(d+2)\log(d+2)\cdot \Gamma_{s-1},\;\; s = 2, \ldots, m.$$
Combining these inequalities, we have

$$\Gamma_m \leq \left[ (d+2)^2\cdot\sum_{k=0}^{m-2} {\left( (d+2)\log(d+2) \right)^k} \right] +\left( (d+2)\log(d+2) \right)^{m-1}\Gamma_1.$$

Since the sum in brackets is a finite geometric series with $m - 1$ terms and $\Gamma_1 \leq d^2$, we have
$$\Gamma_m \leq (d+2)^2\left( \frac{((d+2)\log(d+2))^{m-1}-1}{(d+2)\log(d+2)-1} \right) +((d+2)\log(d+2))^{m-1}\cdot d^2.$$
So we obtain $\Gamma_m \leq (d+2)^{m+1}(\log(d+2))^{m-1}$.
\end{proof}


\begin{theorem} \label{thm:outputs}
Let $\Delta = (g_1,\ldots,g_m) \subset k[x_1,\ldots,x_n]$ be a regular chain of height at most $d$ ($d > 1$). Let $l:=n-m$, and assume that the following conditions are satisfied for every $s=1,\ldots,m$:
\begin{enumerate}[itemsep=2pt]
\item $\lead(g_s) =x_{l+s}$,
\item $\text{lc}(g_s) \in k[x_1,\ldots,x_l]$,
\item $g_s$ is reduced modulo $\Delta_{s-1} = (g_1,\ldots,g_{s-1})$.
\end{enumerate}
Let $M(\Delta)$ be the multiplication table for the algebra $A(\Delta)$. For $f,h \in A(\Delta)[x_{n+1},\ldots,x_{n+c}]$, denote $d_f:=\height(f)$ and $d_h:=\height(h)$.
Then for each polynomial $p$ occurring in the computation of $\unmixed_m^l(\Delta,M(\Delta),f,h)$ (see Algorithm~\ref{alg:Unmixed_m_l}), we have:
\[\height(p) \leq 5.2 \cdot 242^m (d^2+2d)^m d^{\frac{1}{2}m(m+1)} \left( \text{max} \{d,d_f,d_h\}+7(d+2)^m[\log (d+2)]^{m-1}\right) \log d.\]
\end{theorem}

\begin{proof}
Since for the case $m = 1$ unmixed representation can be obtained simply by taking square-free part of the corresponding polynomial (see~\citep[p. 124]{SzantoThesis}), in what follows we assume that $m > 1$.
Let 
\[\left \{ (\Delta_1,M(\Delta_1)),\ldots,(\Delta_r,M(\Delta_r)) \right \}:=\unmixed_m^l(\Delta,M(\Delta),f,h)\]
be the output of the algorithm $\unmixed_m^l$ applied to $(\Delta,M(\Delta),f,h)$. 
Assume that $\Delta_j=(g_{1,j},\ldots,g_{m,j})$ for $j=1,\ldots,r$. For each $s=1,\ldots,m$, we denote
\begin{equation}\label{eq:d_tilde_def}
\tilde{d}_s:=\max \left \{ \deg_{x_{l+s}} \left( g_{s,j} \right) \,|\, j=1,\ldots,r \right \}.
\end{equation}

The computation of $\unmixed_m^l$ has a tree structure. 
Consider a path of the computation tree with successive recursive calls:
\[\unmixed_m^l(\Delta_m,M(\Delta_m),f_m,h_m),\ldots,\unmixed_0^l(\Delta_0,M(\Delta_0),f_0,h_0)\]
where $f_m=f$, $h_m=h$ and $f_s$ and $h_s$ are computed from $(\Delta_{s+1},M(\Delta_{s+1}),f_{s+1},h_{s+1})$ for each $s=0,\ldots,m-1$ as described in Step~\ref{alg:unmixed_three} of Algorithm~\ref{alg:Unmixed_m_l} and \citep[p. 128]{SzantoThesis}. 
First we estimate the height of the input at each level.

\begin{lemma} \label{lem:Input(s)}
Let $\Input(s):=\max \{d,\height(f_s),\height(h_s)\}$ for every $s=0,\ldots,m$. Then 
\[\Input(s) \leq (6d)^{m-s} \left( \Input(m) + 7(d+2)^m(\log (d+2))^{m-1} \right).\]
\end{lemma}

\begin{proof}
We give an inductive analysis to obtain a bound on $\Input(s)$.  
For $s=m$, there is nothing to do.  
So we start with $s = m - 1$ and consider the heights of $f_{m-1}, h_{m-1}$.
Computation of these polynomials from the data of level $m$ in Step~\ref{alg:unmixed_three} of Algorithm~\ref{alg:Unmixed_m_l} can be summarized as follows (see also~\citep[p.~127-128]{SzantoThesis}):
\begin{enumerate}
	\item Compute the $j$-th sub-resultants
	$$
	\varphi_k^{(j)}(y,z):=\Res_{x_n}^{(j)} \left( g_m, f_m+\sum_{l=1}^k 	g_m^{(l)}y^{l-1}+z h_m \right),
	$$
	for $1 \leq k \leq d$ and $0 \leq j \leq d$. Here $y, z$ are new variables (i.e. different from the ones which $g_m, f_m, h_m$ are polynomials in).
	\item For each $1 \leq i \leq d$ and $\mathbf{v} = (v_1,\ldots,v_6) \in \mathbb{Z}_{\geqslant 0}^6$, where $0 \leq v_t \leq d$ for $1 \leq t \leq 6$, 
	\begin{enumerate}
		\item define the polynomial $\phi_{i,\mathbf{v}}(y,z,w)$ to be a linear combination of polynomials 
        \[
        \varphi_{i-1}^{(u_1)}(y,0),\;\; \varphi_{i}^{(u_2)}(y,0),\;\; \varphi_{i+1}^{(u_3)}(y,0),\;\; \varphi_{i-1}^{(u_4)}(y,z),\;\; \varphi_{i}^{(u_5)}(y,z),\;\; \varphi_{i+1}^{(u_6)}(y,z)
        \]
        for all 
        $u_1, \ldots, u_6$ 
        such that $u_i < v_i$ for $1 \leq i \leq 6$
        by using the powers of a new variable $w$.
		\item define 
        \[
        \psi_{i, \mathbf{v}}(y,z):=\varphi_{i-1}^{(v_1)}(y,0) \cdot \varphi_{i}^{(v_2)}(y,0) \cdot \varphi_{i+1}^{(v_3)}(y,0) \cdot \varphi_{i-1}^{(v_4)}(y,z) \cdot \varphi_{i}^{(v_5)}(y,z) \cdot \varphi_{i+1}^{(v_6)}(y,z).
        \]
		\item reduce $\phi_{i,\mathbf{v}}$ and $\psi_{i,\mathbf{v}}$ with respect to $\Lambda$.
        \item Set $f_{m - 1} := \phi_{i, \mathbf{v}}$ and $h_{m - 1} := \psi_{i,\mathbf{v}}$ for this choice of $i, \mathbf{v}$.
	\end{enumerate}
\end{enumerate}

Note that new variables $y,z$ and $w$ were introduced. 
In Algorithm~\ref{alg:Unmixed_m_l}, all new introduced variables are denoted by $x_{n+1},\ldots,x_{n+c}$.
Here we use names $y$, $z$, and $w$ for notational simplicity.

In order to bound the heights of $f_{m - 1}$ and $h_{m - 1}$, 
we bound the heights of the subresultants $\varphi_k^{(j)}(y,z)$.  
In the computation of a bound for the heights of the subresultants, the largest bound will be a bound for the $0$-th subresultant,
because higher ones are obtained by deleting rows and columns of the Sylvester matrix, whose determinant produces the $0$-th subresultant.

Since we are taking subresultants with respect to $x_n$, all the entries of the Sylvester matrix are polynomials in $x_1, x_2, \ldots, x_{n-1}$.  
In particular, this means that their degrees in $x_{l+i}$ are less than $d_i$ for all $1 \leq i < m$.
Size of this matrix is at most $d_m+d_m = 2d_m$.  
The first $d_m$ is because $\deg_{x_n} g_m = d_m$.  
The second $d_m$ is because $f, h$ are reduced with respect to $\Delta$.  

Since $f_{m-1}, h_{m-1}$ must be reduced modulo $\Delta_{m-1}$, we will carrying out all operations in $A(\Delta_{m-1})$.
One can see that the bound for the height of $h_{m-1}$ that we will obtain is larger than a similar computation would produce for $f_{m-1}$.  
So we focus on getting a bound for the height of $h_{m-1}$, thereby obtaining a bound for $\Input(m-1)$.  
In fact, our technique will give us a bound for $\Input(s)$ in terms of $\Input(s+1)$.

Since the computation of $h_{m-1}$ involves a multiplication of six evaluated subresultants, we apply Proposition~\ref{prop:HtProd} to the sixth power of the $0$th subresultant (as described above) in two stages:
\begin{enumerate}[noitemsep, topsep=0pt]
\item
For the first stage, note that each term of the sixth power of the 0-th subresultant is a product of $12d_m$ factors.  
We split these up into two groups: the $6d_m$ factors of any term coming from the  coefficients of $g_m$ (call the product of these $C$) and the rest from the coefficients of $f+\sum_{l=1}^k g_m^{(l)}y^{l-1}+zh$ (call the product of these $D$).  
In this first stage, we need not worry about denominator bounds because all of the factors of $C$ and $D$ are integral elements of $A(\Delta)$.
\item
We then take these two groups of $6d_m$ factors, reduce them, and multiply them.  
In the reduction step, we obtain some denominators in general and so we will need to compute bounds on these.
\end{enumerate}
Assume that the heights of denominators of $C$ and $D$ are bounded by $d'$.
Our two-step analysis of the height of $CD$ using Proposition~\ref{prop:HtProd} yields:
\begin{align*}
\height(CD) & \leq \height(C) + \height(D) + 2\log(2)\cdot (\Gamma(\Delta_{m-1})+d')\\
  & \leq 6d_m \cdot d + 6d_m \cdot \Input(m) + 12d_m\log(6d_m) \cdot \Gamma({\Delta_{m-1}}) + 2\cdot (\Gamma(\Delta_{m-1})+d')\\
  & \leq 6d^2 + 6d \cdot \Input(m) + 12d\log(6d) \cdot \Gamma({\Delta_{m-1}}) + 2\cdot (\Gamma(\Delta_{m-1})+d').
\end{align*}

We need to bound $d'$ by considering the sequence of exponents we obtain on $\text{lc}(g_i)$ when reducing $C, D$ modulo $\Delta_{m-1}$. 
Applying Lemma \ref{lem:denominator} with $\height(C) \leq 6d^2=:t$, we have
\begin{align*}
d' \leq \sum\limits_{i=1}^{m-1} {6d^2 (d+1)^{m-1-i} \cdot d} = 6d^2 (d+1)^{m-1}-6d^2.
\end{align*}

\noindent
Therefore
\begin{multline*}
\height(h_{m-1})  \leq 6d^2 + 6d \cdot \Input(m) + 12d\log(6d) \cdot \Gamma({\Delta_{m-1}}) + \\ + 2\cdot \left(\Gamma(\Delta_{m-1})+6d^2(d+1)^{m-1} - 6d^2\right).
\end{multline*}

\noindent
As a result, we have
\begin{align*}
\Input(m-1) \leq \Gamma({\Delta_{m-1}}) \cdot (12d\log(6d)+2) + 6d \cdot \Input(m) + 12d^2(d+1)^{m-1}.
\end{align*}

\noindent
Moreover, we can obtain a bound for $\Input(s)$ in term of $\Input(s+1)$ in a similar way. 
In particular, we have
$$\Input(s) \leq \Gamma({\Delta_{s}}) \cdot (12d\log(6d)+2) + 6d \cdot \Input(s+1) + 12d^2(d+1)^{s}$$
for every $s=0,\ldots,m-1$. 
Due to Lemma~\ref{lem:boundTheStrucConstant} 
\[
\Gamma({\Delta_{s}}) \leq (d+2)^{s+1} (\log(d+2))^{s-1}.
\]
Using $d \geq 2$, it can be shown that
\begin{align*}
\frac{12d\log(6d)+2}{(d+2)\log(d+2)} \leq 17 \hspace{.5cm} \text{and \hspace{.2cm}} \frac{12d^2}{(d+1)^2} \leq 12.
\end{align*}
We therefore modify our recursive bound and obtain
\begin{align*}
\Input(s) \leq
17 \cdot (d+2)^{s+2} (\log(d+2))^{s} + 6d \cdot \Input(s+1) + 12(d+1)^{s+2}
\end{align*}
for $s = 0, 1, \ldots, m-1$. 
Thus, $\Input(s)$ does not exceed
\[
(6d)^{m-s} \cdot \Input(m) +17\cdot\sum_{k=0}^{m-s-1} (6d)^k(d+2)^{s+k+2}(\log(d+2))^{s+k} + 12\cdot \sum_{k=0}^{m-s-1} (6d)^k(d+1)^{s+k+2}.
\]
\noindent
Using the formula for geometric series and $d \geq 2$, we can deduce that 
$$
\Input(s) \leq (6d)^{m-s}\left( \Input(m) + 6(d+2)^m(\log(d+2))^{m-1} + 3.1(d+1)^m\right).
$$
Using $d, m \geq 2$ we can further show that $3.1(d + 1)^m \leq (d+2)^m(\log(d+2))^{m-1}$, so the above expression is bounded by \[
(6d)^{m-s}\left(\Input(m) + 7(d+2)^m(\log(d+2))^{m-1}\right).\qedhere\]
\end{proof}

We return to the proof of Theorem \ref{thm:outputs}. 
Using the same notation as in~\citep[p.~141]{SzantoThesis}, we denote by $\Output(s)$ the maximum height of polynomials computed up to level $s$. 
For example, if $s=0$, we have $\Output(0)=\Input (0)$. 

We are going to derive an upper bound for $\Output (m)$ recursively. 
Assume that we have determined $\Output (m-1)$ which is an upper bound for all polynomials computed up to level $m - 1$. 
Let $i \leq d$ and $\mathbf{v} \in \mathbb{Z}_{\geq 0}^6$ such that $0 \leq v_t \leq d$ for every $t = 1, 2, \ldots, 6$. Let $\left( \Lambda_{i,\mathbf{v}}, M(\Lambda_{i,\mathbf{v}}) \right)$ be an arbitrary output after the recursive call at level $m - 1$ for these $i$ and $\mathbf{v}$ (see Steps~\ref{alg:unmixed_three} and~\ref{alg_four} of Algorithm~\ref{alg:Unmixed_m_l}). 
The construction of the corresponding output $\left( \Lambda_{i,\mathbf{v}} \cup \{q_{i, \mathbf{v}}\} ,M(\Lambda_{i,\mathbf{v}} \cup \{q_{i, \mathbf{v}}\}) \right)$ from Step~\ref{alg_four} of Algorithm~\ref{alg:Unmixed_m_l} (see also~\citep[p. 129]{SzantoThesis}) is the following
\begin{enumerate}[itemsep=2pt]
	\item\label{step_one} 
	Compute $d_{t,i,v_t}, \, 1 \leq t \leq 6,$ defined by (see \citep[p. 127]{SzantoThesis})
	\begin{align*}
	&d_{1,i,v_1}:=\ggcd_{x_n} \left(\Lambda_{i,\mathbf{v}} \cup \{g_m\},g'_m,\ldots,g_m^{(i-1)},f_m\right)\\
	&d_{2,i,v_2}:=\ggcd_{x_n} \left(\Lambda_{i,\mathbf{v}} \cup \{g_m\},g'_m,\ldots,g_m^{(i)},f_m\right)\\
	&d_{3,i,v_3}:=\ggcd_{x_n} \left(\Lambda_{i,\mathbf{v}} \cup \{g_m\},g'_m,\ldots,g_m^{(i+1)},f_m\right)\\
	&d_{4,i,v_4}:=\ggcd_{x_n} \left(\Lambda_{i,\mathbf{v}} \cup \{g_m\},g'_m,\ldots,g_m^{(i-1)},f_m,h_m \right)\\
	&d_{5,i,v_5}:=\ggcd_{x_n} \left(\Lambda_{i,\mathbf{v}} \cup \{g_m\},g'_m,\ldots,g_m^{(i)},f_m,h_m \right)\\
	&d_{6,i,v_6}:=\ggcd_{x_n} \left(\Lambda_{i,\mathbf{v}} \cup \{g_m\},g'_m,\ldots,g_m^{(i+1)},f_m,h_m \right)
	\end{align*}
	Generalized gcd ($\ggcd$) is described in \citep[Lemma~3.1.3]{SzantoThesis}.
	
    \item\label{step_two} Compute 
    \[
    \overline{\text{lc} (d_{t,i,v_t})}:= \pinvert_{m - }^l \left( \Lambda_{i,\mathbf{v}},M(\Lambda_{i,\mathbf{v}}), \text{lc} (d_{t,i,v_t}) \right) \text{ for } 1\leq t \leq 6,
    \] 
    where the function $\pinvert^l_m(\Delta, M(\Delta), f)$ for computing the pseudo-inverse of $f$ has the following specification (see also~\citep[Section~3.4]{SzantoThesis})
    \vspace{2mm}
    \begin{description}
      \item[In] 
      \begin{itemize}
        \item[$\Delta$:] a squarefree regular chain in $k[x_1, \ldots, x_{l + m}]$, where $x_{l + 1}, \ldots, x_{l + m}$ are the leaders of $\Delta$;
        \item[$M(\Delta)$:] the multiplication table of $A(\Delta)$ (see Definition~\ref{def:mult_table});
        \item[$f$:] a polynomial in $k[x_1, \ldots, x_{l + m}]$ such that $f \not\in P$ for every $P \in \Ap(\Delta)$;
      \end{itemize}
      \item[Out] $\bar{f} \in k[x_1, \ldots, x_{m + l}]$ such that 
      $\bar{f} \cdot \bar{f} \equiv r \pmod{\Rep(\Delta)}$, where $r \in k[x_1, \ldots, x_l] \setminus \{0\}$.
    \end{description}
    \vspace{2mm}
	
    \item\label{step_three} Compute $\overline{d}_{t,i,v_t}:=\overline{\text{lc} (d_{t,i,v_t})} \cdot d_{t,i,v_t}$ for $1 \leq t \leq 6$.
	
    \item\label{step_four} Compute 
    \[
    p_{i,\mathbf{v}}^{(1)}:=\overline{d}_{1,i,v_1} \cdot \overline{d}_{3,i,v_3} \cdot \overline{d}_{5,i,v_5}^2 \;\;\text{ and } \;\; p_{i,\mathbf{v}}^{(2)}:=\overline{d}_{2,i,v_2}^2 \cdot \overline{d}_{4,i,v_4} \cdot \overline{d}_{6,i,v_6},\] 
    and then $q_{i,\mathbf{v}}$, the result of the pseudo-division $p_{i,\mathbf{v}}^{(1)}$ by $p_{i,\mathbf{v}}^{(2)}$.
	
    \item\label{step_five} Compute the multiplication table $M(\Lambda_{i,\mathbf{v}} \cup \{q_{i, \mathbf{v}}\})$.
\end{enumerate}

We are going to bound the heights of the polynomials appearing in each step.

\hyperref[step_one]{Step~1.} The construction of $\ggcd$ in \citep[Lemma~3.1.3]{SzantoThesis} implies that $\height (d_{t,i,v_t}) \leq \Input (m-1)$ for every $t=1,\ldots,6$.

\hyperref[step_two]{Step~2.} We denote by $D_{m-1}$ the dimension of the algebra $A(\Delta)$ over $k$. 
Then $D_{m-1}=\prod\limits_{i=1}^{m-1} \tilde{d}_i$ (see~\eqref{eq:d_tilde_def}). 
The coefficients of $\overline{\text{lc} (d_{t,i,v_t})}$ are defined as the determinants of matrices of size $D_{m-1} \times D_{m-1}$ (see~\citep[p. 84]{SzantoThesis}). 
Every such matrix has a column of the form $[0,\ldots,0,1]^t$, 
and the entries of the matrix have the height at most 
\[\height (d_{t,i,v_t})+\Gamma (\Lambda_{i,\mathbf{v}}) \leq \Input (m-1)+ \Output (m-1).\]
Therefore
\[\height(\overline{\text{lc} (d_{t,i,v_t})}) \leq (D_{m-1}-1) (\Input (m-1)+ \Output (m-1)).\]

\hyperref[step_three]{Step~3.} Now we compute $\overline{d}_{t,i,v_t}:=\overline{\text{lc} (d_{t,i,v_t})} \cdot d_{t,i,v_t}$. Applying \citep[Proposition~3.3.1, p. 66]{SzantoThesis}, we have
\begin{align*}
 \height (\overline{d}_{t,i,v_t}) & \leq \height (\overline{\text{lc} (d_{t,i,v_t})}) + \height (d_{t,i,v_t})+ 2 \log 2 \cdot \Gamma(\Lambda_{i,\mathbf{v}})&\\
& = D_{m-1} \Input (m-1) + (D_{m-1}+1) \Output (m-1).&
\end{align*}

\hyperref[step_four]{Step~4.} Note that, for each $t=1,\ldots,6$, we have $\deg_{x_n} \overline{d}_{t,i,v_t} = \deg_{x_n} (d_{t,i,v_t}) \leq d$.
Therefore $p_{i,\mathbf{v}}^{(1)}$ and $p_{i,\mathbf{v}}^{(2)}$ are polynomials of degree at most $4d$ in $x_n$. 
By using the matrix representation for the quotient of the pseudo-division algorithm, the coefficients of $q_{i,v}$ are equal to a sum of products of at most $4d$ coefficients of $p_{i,\mathbf{v}}^{(1)}$ or $p_{i,\mathbf{v}}^{(2)}$. 
Each coefficient of $p_{i,\mathbf{v}}^{(1)}$ and $p_{i,\mathbf{v}}^{(2)}$ is a sum of products of $4$ coefficients of $\overline{d}_{t,i,v_t},\, t=1,\ldots,6$. 
Thus, coefficients of $q_{i,\mathbf{v}}$ are sums of products of at most $16d$ coefficients of $\overline{d}_{t,i,v_t},\, t=1,\ldots,6$. 
Note that $\overline{d}_{t,i,v_t}$ are polynomials and are reduced by $\Lambda_{i,\mathbf{v}}$. 
Applying \citep[Proposition~3.3.1, p. 66]{SzantoThesis}, we obtain 
 \begin{align*}
 \height (q_{i,\mathbf{v}}) & \leq 16d \cdot \max\limits_{t=1,\ldots,6} \{\height (\overline{d}_{t,i,v_t})\} + 16d \log (16d) \cdot \Gamma (\Lambda_{i,\mathbf{v}})\\
 & \leq \left( 16dD_{m-1}+16d+16d \log (16d) \right) \Output (m-1)+ 16dD_{m-1} \Input (m-1).
 \end{align*}

\hyperref[step_five]{Step~5.} 
As the last step of the computation at level $m$, we compute the multiplication table $M(\Delta_{i,\mathbf{v}})$ for the algebra $A(\Delta_{i,\mathbf{v}})$, where $\Delta_{i, \mathbf{v}} := \Lambda_{i, \mathbf{v}} \cup \{q_{i, \mathbf{v}}\}$. 
We already know that the height of any entry in the multiplication table $M(\Lambda_{i,\mathbf{v}})$ is at most $\Output (m-1)$. 
In order to obtain an upper bound for the heights of coefficients in $M(\Delta_{i,\mathbf{v}})$, we need to estimate the height of the remainder in the pseudo division of $x_{l+1}^{\alpha_1} \cdot \ldots \cdot x_n^{\alpha_m}$ by $q_{i,\mathbf{v}}$, where $0 \leq \alpha_s \leq 2 \deg_{x_{l+s}}(g_{s})-2,\, 1 \leq s \leq m$. 
Note that $q_{i,\mathbf{v}}$ is reduced modulo $\Lambda_{i,\mathbf{v}}$, and that $\deg_{x_n} q_{i,\mathbf{v}} \leq \tilde{d}_m$. 
By using the matrix representation of the remainder in the pseudo-division algorithm (see~\nameref{sec:appendix}), the remainder obtained when we divide $x_{l+1}^{\alpha_1} \cdot \ldots \cdot x_n^{\alpha_m}$ by $q_{i,\mathbf{v}}$ is equal to a sum of products of at most $\tilde{d}_m+2$ integral elements in $A(\Lambda_{i,\mathbf{v}})$. 
Therefore,
 \[\Gamma (\Delta_{i,\mathbf{v}}) \leq (\tilde{d}_m+2) \height (q_{i,\mathbf{v}})+(\tilde{d}_m+2) \log (\tilde{d}_m+2) \Gamma (\Lambda_{i,\mathbf{v}}).\]

This is also an upper bound for all polynomials computed up to level $m$. In other words,
\begin{align*}
\Output(m) & \leq  (\tilde{d}_m+2) \left( 16dD_{m-1}+16d+16 d \log (16d) + \log (\tilde{d}_m+2) \right) \Output (m-1)+ \\ 
& \hspace{3cm} + 16 d D_{m-1}(\tilde{d}_m+2) \Input (m-1).
\end{align*}

We note that the computations are in the algebra $A(\Delta)$. Therefore we always have 
\begin{equation}\label{eq:bound_for_remark}
\tilde{d}_i \leq d \text{ for every } i = 1,\ldots,m.
\end{equation}
Thus $\Output(m)$ does not exceed
 \[
 (d+2)(16d^m+16d \log (32d)+\log (d+2)) \Output(m-1)  + \\ + 16(d+2)d^m \Input(m-1).
\]
A similar argument shows that $\Output(s)$ does not exceed
\begin{equation} \label{equ:Output(s)}
\Output(s) \leq (d+2)(16d^s+16d \log (32d)+\log (d+2)) \Output(s-1) + \\ + 16(d+2)d^s \Input(s-1)
\end{equation}
for every $s = 1,\ldots,m$. 
Lemma \ref{lem:Input(s)} implies that 
\[\Input(0) \leq I_0 := (6d)^{m} \left( \max\{d,d_f,d_h\}+11(d+2)^m(\log (d+2))^{m-1} \right)\]
and
\[\Input(s-1) \leq (6d)^{-s+1} I_0.\]
Using this notation in \eqref{equ:Output(s)}, we see that
\begin{equation}
(6^{s} \Output(s)) \leq C(s) (6^{s-1} \Output(s-1)) + 96d(d+2) I_0
\end{equation}
where
\begin{equation} \label{equ:C(s)}
C(s):=6(d+2)(16d^s+16d \log (32d)+\log (d+2)).
\end{equation}
Now we unfold this recursion and rewrite $6^m \Output(m)$ using $6^{m-1} \Output(m-1)$ and so on, we see that 
\begin{align} \label{equ:Output(m)*6^m}
6^m \Output(m) & \leq \left( \prod\limits_{s=1}^{m}{C(s)} \right) \cdot \Output(0) + 96 d(d+2) I_0 \sum\limits_{s=2}^{m}{\prod\limits_{i=s}^{m}{C(i)}} \nonumber \\
& = \left( \prod\limits_{s=1}^{m}{C(s)}+ 96 d(d+2) \sum\limits_{s=2}^{m}{\prod\limits_{i=s}^{m}{C(i)}}\right) \cdot I_0.
\end{align}

We simplify \eqref{equ:Output(m)*6^m} by applying Lemma~\ref{lem:C(s)}. In particular, we have:
\begin{align*}
6^m \Output(m) < 5.2 \cdot (242(d+2))^m \cdot d^{\frac{1}{2}m(m+1)} \cdot \log d \cdot I_0.
\end{align*}
The obtained inequality after canceling the factor $6^m$ from both sides is exactly the inequality we need to prove.
\end{proof}

\begin{theorem} \label{thm:degreeBound}
Let $F:=\{f_0,f_1,\ldots,f_r\} \subset k[x_1,\ldots,x_n]$ be a set of polynomials of degree at most $d$. Let $m$ be the maximum codimension of prime components of $\sqrt{( F )}$. 
Then the degree of any polynomial $p$ appearing in the output of Algorithm~\ref{alg:SzantoAlgorithm} applied to $F$ or during the computation does not exceed
\begin{equation} \label{inequa:DegreeBound}
B(m,d):= 5.2n \cdot 242^m (d^{2m}+2d^m)^m d^{\frac{1}{2}m^2(m+1)} \left( \max \{d^m,r\}+7(d^m+2)^m(\log (d^m+2))^{m-1}\right) \log d^m.
\end{equation}

In particular, in case $r$ is not too large, for instance if $r \leq d^m$, we have
\begin{align*}
\deg p \leq n d^{(\frac{1}{2}+\epsilon) m^3}
\end{align*}
where $\epsilon=\epsilon(m,d)$ is a decreasing function such that $\epsilon (m,d)<5$ for every $d \geq 2,\, m \geq 2$, and $\lim\limits_{m \to \infty} \epsilon(m, d) = 0$ for all $d$.
\end{theorem}

\begin{remark}\label{rem:linear_comb}
    \citep[Lemma~3]{JeronimoSabia} implies that $f_0, \ldots, f_r$ can be replaced by their $n + 1$ generic linear combinations, so one can achieve $r \leq n$.
\end{remark}

\begin{proof}
By \citep[Corollary~4.1.5, p.~115]{SzantoThesis}, for every $\Delta \in \Sigma (F)$ computed in Step~\ref{alg_one} of Algorithm~\ref{alg:SzantoAlgorithm}, the height of polynomials in $\Delta$ is at most $d^{|\Delta|} \leq d^m$.

At Step~\ref{alg_two} of Algorithm~\ref{alg:SzantoAlgorithm}, for each $\Delta \in \Sigma (F)$, we compute the multiplication table $M(\Delta)$. 
Step~\ref{alg_three} of Algorithm~\ref{alg:SzantoAlgorithm} is a computation of 
\[
\mathcal{U}(\Delta):= \unmixed_{|\Delta|}^{n-|\Delta|}(\Delta,M(\Delta),f,1) \text{ for every } \Delta \in \Sigma(F)
\] where $f=f_0+yf_1+\ldots+y^r f_r \in k[x_1,\ldots,x_n,y]$. 
Note, that for each $\Delta \in \Sigma(F)$, we have $|\Delta| \leq m$.

By Theorem \ref{thm:outputs}, for every polynomial $p$ occurring in the computation of $\mathcal{U}(\Delta)$, we have

\begin{equation*}
\height (p) \leq \frac{1}{n} B(|\Delta|,d).
\end{equation*}

Since $B(m, d)$ is monotonic in $m$ and $|\Delta| \leq m$, this implies \eqref{inequa:DegreeBound}.

In case $r \leq d^m$, we have $\max \{r,d^m\}=d^m$. 
Direct computation shows that the right hand side of \eqref{inequa:DegreeBound} can be bounded by  $\deg p \leq n d^{(\frac{1}{2}+\epsilon) m^3}$, where 
\[\epsilon=\epsilon(m,d):=\frac{ \text{log}_{d} \left( \frac{1}{n} B(m,d) \right)}{m^3}-\frac{1}{2}\]
which is a decreasing function with $\epsilon (m,d) < 5$ for every $d \geq 2,\, m \geq 2$.
Moreover, $\lim\limits_{m \to \infty} \epsilon(m, d) = 0$ for all $d$.
\end{proof}

\begin{remark} \label{rem:IncreasingDegreeBound}
Unlike \citep[Theorem~4.1.7, p. 118]{SzantoThesis}, the height of polynomials occurring in the computations is bounded by $d^{O(m^3)}$. 
In general, Algorithm \ref{alg:SzantoAlgorithm} might produce a redundant unmixed decomposition for a given algebraic set. 
Moreover, it can output varieties defined by regular chains whose irreducible components are not the irreducible components of the initial algebraic set (see~Example~\ref{ex:redundant}).
Therefore the inequality (4.13) in \citep[p. 121]{SzantoThesis} is not necessarily true in general.
Instead of it we use~\eqref{eq:bound_for_remark} in order to bound $\tilde{d}_i$.
The right-hand side of~\eqref{eq:bound_for_remark} is $d^m$ in terms of the input data of Algorithm~\ref{alg:SzantoAlgorithm}, and this makes our bound $d^{O(m^3)}$.
\end{remark}

\begin{lemma} \label{lem:C(s)}
Consider $C(s)$ defined as (see also~\eqref{equ:C(s)})
\[
C(s):=6(d + 2)(16d^s + 16d \log (32d) + \log (d + 2)).
\]
Then we have:
\[\prod\limits_{s=1}^{m}{C(s)} \leq \frac{678\cdot 387}{242^2} \cdot (242(d+2))^m \cdot  d^{\frac{1}{2}m(m+1)} \log d,\,\,\, \text{and}\]
\[\sum\limits_{s=2}^{m}{\prod\limits_{i=s}^{m}{C(i)}} \leq  \frac{387\cdot 4}{967} \cdot (242(d+2))^{m-1}  \cdot d^{\frac{1}{2} m(m+1)-1}.\]
\end{lemma}

\begin{proof}
Using $d \geq 2$, we can verify the following inequalities by direct computation
\begin{equation*}
C(s) \leq \left \{
\begin{aligned}
242 (d+2) d^s \text{ if } s > 2,\\
387 (d+2) d^s \text{ if } s = 2,\\
678 (d+2) d^s \log d \text{ if } s=1.
\end{aligned}
\right.
\end{equation*}
This immediately implies the first inequality in the lemma. For the second one:
\begin{align*}
\sum\limits_{s=2}^{m}{\prod\limits_{i=s}^{m}{C(i)}} &\leq \frac{387}{242}\sum\limits_{s=2}^{m}{(242(d+2))^{m-s+1} \cdot d^{s+(s+1)+ \ldots + m}}
 \\ &\leq \frac{387}{242} d^{\frac{1}{2} m(m+1)-1} \sum\limits_{s=1}^{m-1}{(242(d+2))^{s}}\\
 &\leq \frac{387}{242} d^{\frac{1}{2} m(m+1)-1} \cdot (242(d+2)^{m-1} \cdot \frac{(242(d+2))}{(242(d+2))-1} \\ &\leq \frac{387 \cdot 4}{967} \cdot d^{\frac{1}{2} m(m+1)-1} \cdot (242(d+2))^{m-1}.
 \qedhere
\end{align*}
\end{proof}


\section{Bound for the number of components}
\label{sec:components}

In this section, we study the number of components in the output of Sz\'ant\'o's algorithm.

\begin{theorem}\label{thm:BoundComponents}
Let $F \subset k[x_1,\ldots,x_n]$ be a finite set of polynomials of degree at most $d$. Let $m$ be the maximum codimension of prime components of $\sqrt{( F )} \subseteq k[x_1,\ldots,x_n]$. Then the number of unmixed components in the output of Algorithm \ref{alg:SzantoAlgorithm} applied to $F$ is at most
$$\binom{n}{m}\left( (m + 1)d^m+1\right)^m.$$
\end{theorem}

\begin{proof}
Since the degree of the given polynomials is at most $d$, so is their height. 
Step~\ref{alg_one} of Algorithm \ref{alg:SzantoAlgorithm} produces a set $\Sigma (F):=\{\Delta_{\mathbf{i}}\,|\, \mathbf{i} \subsetneq [n]\}$ of regular chains such that for every prime component $P$ of $\sqrt{( F )}$, we have
\[(\dim P = |\mathbf{i}| \text{ and } P \cap k[x_i \;|\; i \in \mathbf{i} ]=0) \Rightarrow \Rep(\Delta) \subseteq P.\]
Due to~\citep[Theorem~4.4]{Hubert1}, the number of elements in a regular chain $\Delta$ is equal to the codimension of the ideal $\Rep(\Delta)$. 
Therefore the number of regular chains in $\Sigma(F)$ is not larger than the number of proper subsets of $[n]$ which has cardinality at most $m$.

In Step~\ref{alg_three}, we use the function $\unmixed$ to transform each  regular chain $\Delta \in \Sigma(F)$ to the set
\[\mathcal{U}(\Delta):=\unmixed_{|\Delta|}^{n-|\Delta|}(\Delta,M(\Delta),f,1)\]
of squarefree regular chains (see Algorithm~\ref{alg:Unmixed_m_l}). 
Thus the number of squarefree regular chains in the output is 
\begin{align*}
M(n,m,d)  :=  \left | \bigcup\limits_{\Delta \in \Sigma(F)} \mathcal{U}(\Delta) \right |  \leq \sum\limits_{\Delta \in \Sigma(F)} \left | \mathcal{U}(\Delta) \right |.
\end{align*}

We fix a regular chain $\Delta = ( g_1, \ldots, g_s )$ of codimension $s$.
The collection of squarefree regular chains in the output of $\unmixed_{|\Delta|}^s$ is simple, meaning that any two distinct unmixed components have no common irreducible components (see \citep[page 124]{SzantoThesis}).  
Since all the components of $\Rep(\Delta)$ are of codimension $s$, $|\mathcal{U}(\Delta)|$ is bounded from above by the degree of $\Rep(\Delta)$.
Due to the definition of $\Rep(\Delta)$, we have $\Rep(\Delta) \supset (\Delta)$.
Moreover, since $V(\Delta)$ and $V(\Rep(\Delta))$ coincide outside the zero set of the product of the initials of $\Delta$, every irreducible component of $V(\Rep(\Delta))$ is an irreducible component of $V(\Delta)$.
Hence, the degree of $\Rep(\Delta)$ does not exceed the sum of degrees of irreducible components of $V(\Delta)$.
The latter can be bounded by $\deg g_1\cdot \ldots\cdot \deg g_s$ due to~\citep[Theorem~1]{Heintz}.
The proof of \citep[Corollary 4.1.5]{SzantoThesis} implies that every $g_i$ depends on at most $s + 1$ variables, so 
\[
\deg g_i \leq (s + 1)\height g_i \leq (s + 1)d^s.
\]
Therefore 
\[
|\mathcal{U}(\Delta)| \leq (s + 1)^s d^{s^2} \leq \left( (m + 1)d^m \right)^s.
\]

Since for each $s=1,\ldots,m$, there are $\binom{n}{s}$ squarefree regular chains in $\Sigma(F)$ of cardinality $s$,
\begin{align*}
M(n,m,d) \leq \sum_{s=1}^m \binom{n}{s} \left( (m + 1)d^m \right)^s.
\end{align*}
Since $\binom{n}{s} \leq \binom{n}{m} \cdot \binom{m}{s}$, we have that
$M(n,m,d) \leq \binom{n}{m}\left( (m + 1)d^m+1\right)^m$.
\end{proof}

\begin{acknowledgment}
We are grateful to Agnes Sz\'ant\'o and Alexey Ovchinnikov for discussions related to this work and to the anonymous referees for comments and suggestions, which helped us improve the paper significantly.
\end{acknowledgment}

\section*{Appendix}
\label{sec:appendix}

The following results on matrix representations of pseudoremainders are used in Section~\ref{sec:degrees}.  They are mentioned and used in \citep[Section 3.3]{SzantoThesis}.  
We include here a shortened and refined version of them.

Let $f \in k[x_1, x_2, \ldots, x_l], g \in k[x_1, x_2, \ldots, x_n]$ with $k$ a field and $l \geq n$.  
We wish to describe the pseudoremainder of $f$ by $g$ with respect to $x_n$ in matrix form.  
More specifically, we wish to describe this pseudoremainder when $\text{deg}_{x_n}(g)=d$ and $\text{deg}_{x_n}(f) \leq 2d-2$, (the application in mind being computing the structure constants for $A(\Delta)$, see Definition~\ref{def:mult_table}).  
We will allow the degree of $f$ to go up to $2d-1$ in fact.
We first write $f$ and $g$ as univariate polynomials in $x_n$ with coefficients $k[x_1,\ldots,x_{n-1},x_{n+1}, \ldots, x_l]$:
$$
f = f_0+f_1x_n+\dots+f_{2d-1}x_n^{2d-1}, \quad g = g_0+g_1x_n+\dots+g_dx_n^d.
$$
Note that the difference between the degrees in $x_n$ of $f$ and $g$ is $d-1$.  Thus, the pseudoremainder equation we consider (in scalar form) is $g_d^df = gq+r$
where the degrees in $x_n$ of $r, q$ are less than $d$.  
Writing $q$ and $r$ as we wrote $f, g$ above and substituting these expressions into the pseudoremainder equation, we obtain:
$$
g_d^d(f_0+\ldots+f_{2d-1}x_n^{2d-1}) = (g_0+\ldots+g_dx_n^d)(q_0+\ldots+q_{d-1}x_n^{d-1})+
r_0+\ldots+r_{d-1}x_n^{d-1}.
$$

Comparing coefficients of the powers of $x_n$ from $d$ to $2d-1$, we obtain the following linear system
$$\begin{pmatrix} 
g_d & 0 & 0 & \dots & 0 \\
g_{d-1} & g_d & 0 & \dots & 0 \\
\dots & \dots & \dots & \dots & \dots\\
g_1 & g_2 & \dots & \dots & g_d 
\end{pmatrix} 
\begin{pmatrix}
q_{d-1} \\
q_{d-2} \\
\dots \\
q_0
\end{pmatrix} = 
\begin{pmatrix}
f_{2d-1} \\
f_{2d-2} \\
\dots \\
f_d
\end{pmatrix} g_d^d.$$
We write the system above as $G_d\mathbf{q} = \mathbf{f}^{\uupper}g_d^d$.  Since $g_d \neq 0$ (as $g$ is assumed to have degree $d$ in $x_n$), we can find the coefficients of the desired quotient by inverting $G_d$.

Since $r = g_d^df - qg$, after substituting we obtain one more linear system
$$\begin{pmatrix}
r_{d-1} \\
r_{d-2} \\
\dots \\
r_0
\end{pmatrix} = 
g_d^d\begin{pmatrix}
f_{d-1} \\
f_{d-2} \\
\dots \\
f_0
\end{pmatrix} - 
\begin{pmatrix} 
g_0 & g_1 & \dots & \dots & g_{d-1} \\
0 & g_0 & g_1 & \dots & g_{d-2} \\
\dots & \dots & \dots & \dots & \dots\\
0 & 0 & \dots & \dots & g_0
\end{pmatrix}
\begin{pmatrix}
q_{d-1} \\
q_{d-2} \\
\dots \\
q_0
\end{pmatrix}.
$$
We write this system as $\mathbf{r} = g_d^d \mathbf{f}^{\llower} - G_0\mathbf{q}$.
Combining with the equation for $\mathbf{q}$, we obtain $$\mathbf{r} = g_d^d \mathbf{f}^{\llower} - g_d^dG_0G_d^{-1}\mathbf{f}^{\uupper}.$$
To count multiplications in the formula for the pseudoremainder, we re-express $G_d^{-1}$ using Cramer's Rule: $G_d^{-1} = g_d^{-d}\cdot \adj(G_d)$
where $\adj(G_d)$ denotes the adjugate of $G_d$, (i.e. its matrix of cofactors transposed).  So we have
$\mathbf{r} = g_d^d \mathbf{f}^{\llower} - G_0\cdot \adj(G_d)\mathbf{f}^{\uupper}.$

Observe that the entries of $\adj(G_d)$ are sums of products of $d-1$ entries of $G_d$.

\end{document}